\documentclass{amsart}

\usepackage{epic}
\usepackage{fullpage}

\title{An elliptic hypergeometric beta integral transformation}
\author{Fokko van de Bult \\ California Institute of Technology}

\newcommand{\II}{I\!I}

\newcommand{\III}{I\!I\!I}

\newtheorem{thm}{Theorem}[section]
\newtheorem{cor}[thm]{Corollary}
\newtheorem{lemma}[thm]{Lemma}

\begin{document}

\begin{abstract}
In this article we prove a new elliptic hypergeometric integral identity. It previously appeared (as a conjecture) in articles by Rains \cite{REL}, and Spiridonov and Vartanov \cite{SV}. Moreover it gives a different proof of an identity in another article by Rains \cite{Rtrafo}.
We also give some basic hypergeometric and classical limits of this identity. The classical limit gives identities (some known, some new) between generalizations of the Selberg integral.
\end{abstract}

\maketitle

\section{Introduction}
The subject of elliptic hypergeometric functions is relatively new. It started when Frenkel and Turaev \cite{FT} obtained the first
elliptic hypergeometric series evaluation. Even though the theory of elliptic hypergeometric functions (series and integrals) has not yet been as well studied as the theory of classical and basic hypergeometric functions, steadily more and more is known about these functions. A recent overview of these results is given in \cite{Spiressays}. 

In this short article we prove a multivariate elliptic hypergeometric integral identity. It first appeared in 
a paper by Eric Rains \cite[(8.14)]{Rtrafo}, though in a non-explicit form. In this paper Rains defines certain integral and difference operators. He proves that two integral operators commute on a certain space of functions, by giving a complete set of eigenfunctions (the biorthogonal functions of that paper) for these integral operators. A direct proof of the commutativity requires an elliptic hypergeometric transformation, which is a special case of our main theorem, Theorem \ref{thmmain}. 

Our main theorem appeared again as Conjecture 1 in another paper by Rains \cite{REL}. In this case it is obtained as an integral analogue of a 
elliptic hypergeometric series identity (which Rains proves). This series identity is a special case of the complementation symmetry for 
the skew interpolation function defined there. This series identity can also be seen as a kind of dual Karlsson-Minton sum. 

Very recently, the conjecture also appeared in a paper by Spiridonov and Vartanov \cite{SV}. In their case the equation we prove provides the equality between the superconformal indices of two (conjecturally) dual supersymmetric quantum field theories. 

The latter half of this article derives some basic hypergeometric integral identities which can easily be obtained as limits from the main theorem. For the limit $p\to 0$ the resulting identities are transformations between multivariate $q$-beta integrals. These integral transformations appear to be new. The integrals obtained are similar in flavor to integrals associated to root systems considered by Gustafson (for example in \cite{Gustafson}). The limit $q\to 1$ gives certain generalizations of the Selberg integral, which were studied before in for example \cite{Yan}. One of the classical limits is a known formula, which generalizes Euler's transformation for beta integrals.

The article is organized as follows. In Section \ref{secprelim} we define the notations and give the key theorems used in the proof of the main theorem. In Section \ref{secmain} we state and prove our main elliptic hypergeometric identity. The last two sections cover some basic hypergeometric, respectively classical, limits.

\subsection*{Acknowledgments} I would like to thank E.M. Rains for our interesting discussions. I also like to thank V.P. Spiridonov for his remarks.

\section{Preliminaries}\label{secprelim}
Throughout the article $p$ and $q$ denote complex numbers of modulus less than one (i.e. $|p|,|q|<1$). 

The $q$-shifted factorials and related theta functions are defined as
\[
(x;q)= \prod_{r=0}^{\infty} (1-xq^r), \qquad \theta(x;q) = (x;q) (q/x;q).
\]
The infinite product converges as $|q|<1$. We use the usual abbreviations $(a_1,a_2,\ldots,a_n;q) = \prod_{i=1}^n (a_i;q)$ and $(az^{\pm 1};q) = (az, a/z;q)$.

The elliptic gamma function \cite{Ruijs} is defined as  the doubly infinite product
\[
\Gamma(z):= \Gamma(z;p,q) = \prod_{r,s\geq 0} \frac{1-p^{r+1}q^{s+1}/z}{1-p^rq^sz}.
\]
We will generally only be concerned with the elliptic gamma function, so no confusion should arise with the 
classical Euler gamma function. However in the section on classical limits we do need both gamma functions, so in that section we use $\Gamma_e$ for the elliptic gamma function and $\Gamma_c$ for the classical gamma function.

We will use similar abbreviations for the elliptic gamma function as for the $q$-shifted factorials. The elliptic gamma function
satisfies the difference equations
\[
\Gamma(pz;p,q) = \theta(z;q) \Gamma(z;p,q), \qquad \Gamma(qz;p,q) = \theta(z;p) \Gamma(z;p,q).
\]
and the following reflection equation
\begin{equation}\label{eqrefl}
\Gamma(z) \Gamma(pq/z) = 1.
\end{equation}

Observe the following limits to $q$-shifted factorials
\[
\lim_{p\to 0} \Gamma(z;p,q) = \frac{1}{(z;q)}, \qquad \lim_{p\to 0} \Gamma(pz;p,q) = (q/z;q).
\]

We define the integration kernels
\begin{align*}
\Delta_I^{(n)}(z) & =\frac{1}{\prod_{1\leq i<j\leq n} \Gamma(z_i^{\pm 1} z_j^{\pm 1}) \prod_{i=1}^n \Gamma(z_i^{\pm 2})} \\
\Delta_{\II}^{(n)}(t;z) &= \frac{\prod_{1\leq i<j\leq n} \Gamma(t z_i^{\pm 1} z_j^{\pm 1}) }{\prod_{1\leq i<j\leq n} \Gamma(z_i^{\pm 1} z_j^{\pm 1}) \prod_{i=1}^n \Gamma(z_i^{\pm 2})} 
\end{align*}
and the constants
\[
P_{BC_n} = \frac{(p;p)^n (q;q)^n}{2^n n!}, \qquad P_{A_{n-1}} = \frac{(p;p)^n (q;q)^n}{n!}
\]

The following theorem is the Elliptic Dixon transformation proven by Rains \cite{Rtrafo}.
\begin{thm}\label{thmselberg}
Under the balancing condition $\prod_{r=0}^{2n+2m+3} t_r = (pq)^{m+1}$ we have 
\begin{multline*}
P_{BC_n} \int_{C^n} \Delta_I^{(n)}(z) \prod_{i=1}^n \prod_{r=0}^{2n+2m+3} \Gamma(t_r z_i^{\pm 1})  \frac{dz_i}{2\pi i z_i} 
\\= 
P_{BC_m} \prod_{0\leq r<s \leq 2n+2m+3} \Gamma(t_rt_s)
\int_{C^m} \Delta_I^{(m)}(z)
\prod_{i=1}^m \prod_{r=0}^{2n+2m+3} \Gamma(\frac{\sqrt{pq}}{t_r} z_i^{\pm 1}) \frac{dz_i}{2\pi i z_i} 
\end{multline*}
The contours of the integrals are unit circles for parameters $|\sqrt{pq}|<|t_r|<1$, and otherwise we view it as an identity between the analtytic extensions as meromorphic functions in the $t_r$ of these functions.
\end{thm}

An important special case is obtained by setting $m=0$, in which case the transformation reduces to an evaluation (as the right hand side becomes an integral over 0 variables). This case was conjectured by Van Diejen and Spiridonov \cite{vDS}, the latter author also having proved the special case $n=1$, $m=0$ in \cite{Spir1}. An independent proof of this case is found in \cite{Spir2}. A basic hypergeometric analogue was already given by Gustafson \cite{Gustafson}. The relevant equation is explicitly given by
\[
P_{BC_n} \int_{C^n} \Delta_I^{(n)}(z) \prod_{i=1}^n \prod_{r=0}^{2n+3} \Gamma(t_r z_i^{\pm 1})  \frac{dz_i}{2\pi i z_i} 
= 
 \prod_{0\leq r<s \leq 2n+3} \Gamma(t_rt_s).
\]

The elliptic Selberg evaluation and transformation \cite{Rtrafo} given below are other important elliptic hypergeometric integral identities. The evaluation was again conjectured (and shown to be a consequence of the elliptic Dixon evaluation) in \cite{vDS}.
\begin{thm}\label{thmdixon}
Under the balancing condition $t^{2(n-1)}t_1t_2t_3t_4t_5t_6=pq$ we have
\[
P_{BC_n} \Gamma(t)^n \int_{C^n} \Delta_{\II}^{(n)}(t;z) \prod_{i=1}^n \prod_{r=1}^6 \Gamma(t_rz_i^{\pm 1}) \frac{dz_i}{2\pi i z_i} = 
\prod_{i=0}^{n-1} \prod_{1\leq r<s\leq 6} \Gamma(t^i t_rt_s).
\]
Under the balancing condition $t^{2(n-1)} \prod_{r=1}^8 t_r=(pq)^2$ we have
\begin{multline*}
P_{BC_n} \Gamma(t)^n \int_{C^n} \Delta_{\II}^{(n)}(t;z) \prod_{i=1}^n \prod_{r=1}^8 \Gamma(t_rz_i^{\pm 1}) \frac{dz_i}{2\pi i z_i} \\ = 
\prod_{i=0}^{n-1} \prod_{r=1}^4 \prod_{s=5}^8 \Gamma(t^i t_rt_s)
P_{BC_n} \Gamma(t)^n \int_{C^n} \Delta_{\II}^{(n)}(t;z) \prod_{i=1}^n \prod_{r=1}^4 \Gamma(vt_rz_i^{\pm 1}) \prod_{r=5}^8 \Gamma(t_r/vz_i^{\pm 1}) \frac{dz_i}{2\pi i z_i} 
\end{multline*}
where $v^2 = pq/t^{n-1}t_1t_2t_3t_4 = t^{n-1}t_5t_6t_7t_8/pq$.

The contours in these identities are taken as unit circles if $|t_r|<1$ ($1\leq r\leq 8$) and $|vt_r|<1$ ($1\leq r\leq 4$) and $|t_r/v|\leq 1$ 
($5\leq r\leq 8$), and otherwise we regard it as an equation between the analytic extensions of these integrals (as meromorphic function of $t_r$, $t$, $p$ and $q$).
\end{thm}

\section{The main theorem}\label{secmain}
In this section we will prove Conjecture 1 from \cite{REL}. This conjecture also appeared in Section 7 of \cite{SV} as a 
special case of the equality of the superconformal indices of two quantum field theories related by Seiberg duality.

Concretely, we will prove the following identity between elliptic Selberg integrals
\begin{thm}\label{thmmain}
Under the balancing conditions 
\[
t_0t_1t_2t_3=t^{2+m-n},  \qquad v_{2i}v_{2i+1} = \frac{pq}{t},\quad (0\leq i<k), \qquad  k=m+n
\]
 we have 
\begin{align*}
P_{BC_n} & \Gamma(t)^n \int_{C^n} 
\Delta_{\II}^{(n)}(t;z) 
\prod_{i=1}^n \prod_{r=0}^3 \Gamma(t_r z_i^{\pm 1}) \prod_{r=0}^{2k-1} \Gamma(v_r z_i^{\pm 1}) \frac{dz_i}{2\pi i z_i}
\\ & =
\prod_{i=m+1}^n \prod_{0\leq r<s\leq 3} \Gamma(t_rt_s t^{n-i}) 
\prod_{i=0}^{2k-1} \prod_{r=0}^3 \Gamma(t_r v_i)
\\ & \qquad \times P_{BC_m} \Gamma(t)^{m} \int_{C^{m}} 
\Delta_{\II}^{(m)}(t;z) 
\prod_{i=1}^{m} \prod_{r=0}^3 \Gamma(t t_r^{-1} z_i^{\pm 1}) \prod_{r=0}^{2k-1} \Gamma(v_r z_i^{\pm 1}) \frac{dz_i}{2\pi i z_i},
\end{align*}
where the integration contour is a product of unit circles if $|t|<|t_r|<1$ ($0\leq r\leq 3$), and $|v_r|<1$ ($0\leq r\leq 2k-1$) and we view this as an identity between the analytic extensions of these integrals otherwise.
\end{thm}
Remark that in \cite{Rtrafo} it was shown that the integrals actually define single-valued meromorphic functions for parameters $t, t_r,v_r\in \mathbb{C}$ and $p,q \in \{ z\in \mathbb{C} ~|~ |z|<1\}$, subject to the balancing conditions. 

The proof will consist of creating the circle of identities depicted in Figure \ref{fig1}.
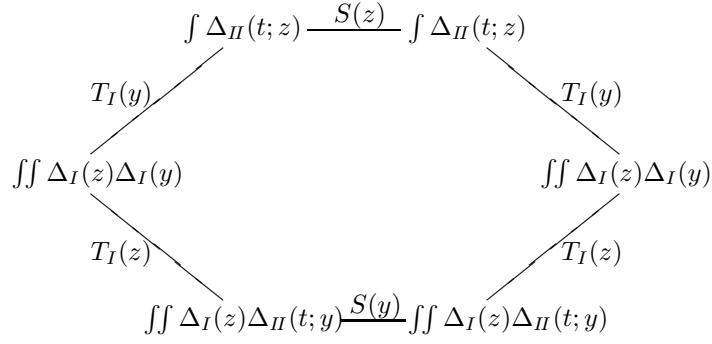
\begin{figure}
\begin{center}
\begin{picture}(260,150)
\put(50,20){$\iint \Delta_{I}(z) \Delta_{\II}(t;y)$}
\put(0,75){$\iint \Delta_{I}(z) \Delta_I(y)$}
\put(65,130){$\int \Delta_{\II}(t;z) $}
\put(150,130){$\int \Delta_{\II}(t;z) $}
\put(200,75){$\iint \Delta_{I}(z) \Delta_{I}(y)$}
\put(150,20){$\iint \Delta_{I}(z) \Delta_{\II}(t;y)$}
\drawline(80,30)(30,70)
\drawline(180,30)(230,70)
\drawline(80,125)(30,85)
\drawline(180,125)(230,85)
\drawline(125,22)(149,22)
\drawline(112,132)(147,132)
\put(208,105){$T_I (y)$}
\put(208,45){$T_I(z)$}
\put(30,105){$T_I(y)$}
\put(30,45){$T_I(z)$}
\put(122,135){$S(z)$}
\put(128,25){$S(y)$}
\end{picture}
\caption{The hexagon of double integral identities}\label{fig1}

\end{center}
\end{figure}
Here an edge labeled $T_I(x)$ indicates two double integrals equated to each other by an application of the Dixon transformation (Theorem \ref{thmselberg}) in the $x$ variable and 
$S(x)$ denotes an application of the conjectured transformation of Theorem \ref{thmmain} in the $x$ variable.
In order to turn this into a proof, we ensure that we use a special case of the theorem with one fewer variables in the bottom identity, than in the top identity. By induction we can then prove the general case.
\begin{proof}
We first prove the identity only for $|n-m|\leq 1$. In the coming calculation we hence assume $n\leq m\leq n+1$. 
By a limiting argument we can subsequently prove the general case. In the following calculations we assume all integration contours are taken as unit circles, which allows us to interchange the $y$ and $z$ integration at will. At the end we will show there exists an open set of parameters such that this is allowed; the general equation then follows by analytic continuation. 

Take parameters as in the statement of the theorem, but write $u:=t_3$. Moreover we write
$b= t^m \sqrt{pqt}/ upq$. 
Let 
\begin{align*}
I:= P_{BC_n} & \Gamma(t)^n \int_{C^n} \Delta_{\II}^{(n)}(t;z) \prod_{i=1}^n \prod_{r=0}^2 \Gamma(t_r z_i^{\pm 1}) \Gamma(u z_i^{\pm 1})
\prod_{r=0}^{2k-1} \Gamma(v_r z_i^{\pm 1})\frac{dz_i}{2\pi i z_i} .
\end{align*}
By viewing part of the integrand as the result of an elliptic Dixon evaluation (i.e. the case of Theorem \ref{thmselberg} with $m=0$) we get 
\begin{align*}
I& =
\frac{P_{BC_n}^2 }{\Gamma(t_0t_1/t,t_0t_2/t,t_1t_2/t) \prod_{r=0}^2 \Gamma(t_r \sqrt{pq/t} /b)}
\int_{C^n} \int_{C^n} 
\Delta_I^{(n)}(z) \Delta_I^{(n)}(y) 
\prod_{1\leq i,j\leq n} \Gamma(\sqrt{t} y_i^{\pm 1} z_j^{\pm 1})
\\ & \qquad \times 
\prod_{i=1}^n \Gamma(u z_i^{\pm 1},b \sqrt{pq/t} z_i^{\pm 1})\prod_{r=0}^{2k-1} \Gamma(v_r z_i^{\pm 1})
\prod_{i=1}^n \prod_{r=0}^{2} \Gamma(t_r/\sqrt{t} y_i^{\pm 1}) 
\Gamma(\sqrt{pq}/b y_i^{\pm 1})
\frac{dz_i}{2\pi i z_i} \frac{dy_i}{2\pi i y_i} 
\end{align*}
The balancing condition for the $y$ integral here is 
\[
\frac{t_0t_1t_2 }{b} \sqrt{\frac{pq}{t^3}} t^n = pq, 
\]
which indeed follows from our conditions. Also the number of parameters of the resulting $y$ integral is $2n+4=2\cdot n + 2\cdot 0 + 4$ as it should be. The $z$ integral is now also of Dixon type, so we can transform the $z$ integral (making sure that the balancing condition is satisfied and the number of parameters is correct) to get 
\begin{align*}
I &=
\frac{P_n P_{k-1} \Gamma(t)^n \prod_{0\leq r<s\leq 2k-1} \Gamma(v_rv_s)\prod_{r=0}^{2k-1} \Gamma(v_r u,v_rb \sqrt{pq/t}) \Gamma(ub \sqrt{pq/t})}{\Gamma(t_0t_1/t,t_0t_2/t,t_1t_2/t) \prod_{r=0}^2 \Gamma(t_r \sqrt{pq/t} /b)}
\\ & \qquad \times \int_{C^n (y)} \int_{C^{k-1} (z)} 
\Delta_I^{(k-1)}(z) \Delta_{\II}^{(n)}(t;y) 
\prod_{i=1}^{n} \prod_{j=1}^{k-1} \Gamma(\sqrt{pq/t} y_i^{\pm 1} z_j^{\pm 1})
\\ & \qquad \qquad \times 
\prod_{i=1}^{k-1} \Gamma(\sqrt{pq}/u z_i^{\pm 1}, \sqrt{t}/b z_i^{\pm 1})\prod_{r=0}^{2k-1} \Gamma(\sqrt{pq} v_r^{-1} z_i^{\pm 1})
\\ & \qquad\qquad  \times 
\prod_{i=1}^n \Gamma(\sqrt{t} u y_i^{\pm 1}) \prod_{r=0}^{2} \Gamma(t_r/\sqrt{t} y_i^{\pm 1}) \prod_{r=0}^{2k-1} \Gamma(\sqrt{t} v_r y_i^{\pm 1})
\frac{dz_i}{2\pi i z_i} \frac{dy_i}{2\pi i y_i} 
\end{align*}
Now we can use the transformation of the theorem in the $y$ integral. The parameters
$\sqrt{t}v_{i}$  satisfy the equation $\sqrt{t} v_{2i} \sqrt{t}v_{2i+1} =  pq$, so by the reflection equation of the elliptic gamma functions these parameters vanish. However now we have $\sqrt{pq/t} z_i \sqrt{pq/t} z_i^{-1} = pq/t$, so these form the new $v_i$ parameters. Thus the number of 
pairs of new ``$v_i$'' parameters equals the number of $z$ parameters equals $k-1$. Moreover 
\[
\sqrt{t} u \frac{t_0}{\sqrt{t}}\frac{t_1}{\sqrt{t}}\frac{t_2}{\sqrt{t}} =
\frac{ut_0t_1t_2}{t} = t^{2+(m-1)-n}.
\]
Thus the $y$-integral is of the form of the theorem with $m$ one lower. 
Applying the theorem in this case we get 
\begin{align*}
I &=
P_{m-1} P_{k-1} \Gamma(t)^{m-1} \prod_{i=m}^n \left( \Gamma(t_0t_1 t^{n-i-1},t_0t_2 t^{n-i-1},t_1t_2 t^{n-i-1})
\prod_{j=0}^2 \Gamma(t_j u t^{n-i})\right) \prod_{0\leq r<s\leq 2k-1} \Gamma(v_rv_s) \\ & \qquad \times \frac{
\prod_{r=0}^{2k-1} \Gamma(v_r u,v_rb \sqrt{pq/t}) \Gamma(ub \sqrt{pq/t})}{\Gamma(t_0t_1/t,t_0t_2/t,t_1t_2/t) \prod_{r=0}^2 \Gamma(t_r \sqrt{pq/t} /b)}
\\ & \qquad \times \int_{C^{m-1} (y)} \int_{C^{k-1} (z)} 
\Delta_I^{(k-1)}(z) \Delta_{\II}^{(m-1)}(t;y)
 \prod_{i=1}^{m-1} \prod_{j=1}^{k-1} \Gamma(\sqrt{pq/t} y_i^{\pm 1} z_j^{\pm 1})
\\ & \qquad \qquad \times 
\prod_{i=1}^{k-1} \Gamma(\sqrt{t}/b z_i^{\pm 1}) \prod_{r=0}^2 \Gamma(t_r \sqrt{pq}t^{-1} z_i^{\pm 1})\prod_{r=0}^{2k-1} \Gamma(\sqrt{pq} v_r^{-1} z_i^{\pm 1})
\\ & \qquad\qquad \qquad   \times \prod_{i=1}^{m-1} \Gamma(\sqrt{t}/u y_i^{\pm 1}) \prod_{r=0}^{2} \Gamma(t\sqrt{t}t_r^{-1} y_i^{\pm 1}) 
\frac{dz_i}{2\pi i z_i} \frac{dy_i}{2\pi i y_i} 
\end{align*}
Next we can again use the Dixon transformation on the $z$ integral.
We obtain the result
\begin{align*}
I &=
P_{m-1} P_{m} \prod_{i=m}^{n-1} \Gamma(t_0t_1 t^{n-i-1},t_0t_2 t^{n-i-1},t_1t_2 t^{n-i-1})
\prod_{i=m+1}^{n} \prod_{j=0}^2 \Gamma(t_j u t^{n-i})  \\ & \qquad \times
\prod_{i=0}^{2k-1} \Gamma(t_0 v_i ,t_1v_i,t_2v_i,u v_i)
 \Gamma(ub \sqrt{pq/t})
\\ & \qquad \times \int_{C^{m-1} (y)} \int_{C^{m} (z)} 
\Delta_I^{(m)}(z) \Delta_I^{(m-1)}(y) 
\prod_{i=1}^{m-1} \prod_{j=1}^{m}  \Gamma(\sqrt{t} y_i^{\pm 1} z_j^{\pm 1})
\\ & \qquad \qquad \times 
\prod_{i=1}^{m} \Gamma(\sqrt{pq/t} b z_i^{\pm 1})\prod_{r=0}^3 \Gamma(tt_r^{-1} z_i^{\pm 1})\prod_{r=0}^{2k-1} \Gamma(v_r z_i^{\pm 1})
 \prod_{i=1}^{m-1} \Gamma(\sqrt{t}/u y_i^{\pm 1}, \sqrt{pq}/b y_i^{\pm 1})
\frac{dz_i}{2\pi i z_i} \frac{dy_i}{2\pi i y_i} 
\end{align*}
Finally we can use an elliptic Dixon evaluation to evaluate the $y$ integral. We therefore obtain
\begin{align*}
I &=
\prod_{i=m+1}^{n} \prod_{0\leq r<s\leq 2} \Gamma(t_rt_s t^{n-i})  \prod_{r=0}^2 \Gamma(t_r u t^{n-i})
\prod_{i=0}^{2k-1} \prod_{r=0}^2  \Gamma(t_r v_i) \Gamma(u v_i)
 P_{m} \Gamma(t)^{m} 
\\ & \qquad \times \int_{C^{m}} \Delta_{\II}^{(m)}(t;z) 
\prod_{i=1}^{m} \Gamma(t/u z_i^{\pm 1})\prod_{r=0}^2 \Gamma(tt_r^{-1} z_i^{\pm 1})\prod_{r=0}^{2k-1} \Gamma(v_r z_i^{\pm 1})
\frac{dz_i}{2\pi i z_i}.
\end{align*}
Comparing with the original equation for $I$, we see that we have proven the theorem for 
the pair $(n,m)$, if we assume the theorem holds for $(n,m-1)$. 

In order for us not to have to worry about the constant interchanging of the $z$ and $y$ integrals 
we chose parameters such that the contour becomes a product of unit circles. 
For this we need the conditions that all parameters (in each integral) are less than 1 in size. Thus we get the bounds
\[
|pq|<|t|<1, \qquad
|pq|<|v_r^2| < 1,\qquad 
|t|<|u^2|<1, \qquad
|t|^2<|t_r|^2<|t|, \qquad 
|t|  < |b|^2 < |t/pq|.
\]
We will now show that these bounds can be satisfied, at least when $n\leq m\leq n+1$.
Let us pick $p$ and $q$ as we wish, norm less than 1. Moreover choose 
$|t|> |pq|^{1/2(n+m)}$. Then we can choose $v_r$'s without problem in the range
$|\sqrt{pq}|<v_r< |\sqrt{pq} /t|$, which ensures their partner is also in this range.
(Note that $|\sqrt{pq} /t|<1$, due to our bound on $|t|$.) If $m=n$ or $m=n+1$ we can 
choose 4 numbers $t_r$ ($r=0,1,2$) and $w$ with $|t|<|t_r|, |w| <|\sqrt{t}|$ such that $t_0t_1t_2w=t^{5/2+m-n}$ 
(because the norm of their product can be anything from $|t|^4$ to $|t|^2$ and the right hand side is 
either $|t|^{5/2}$ or $|t|^{7/2}$). Now $t_0$, $t_1$, and $t_2$ are as chosen and take $u=\sqrt{t}w$ to
satisfy their balancing condition. Finally we have to take $b=t^{m+1/2}/u\sqrt{pq}$ and see that
$|\sqrt{t}| < |t^{m+1/2}/\sqrt{pq}|< |b|< |t^m /\sqrt{pq}| \leq |\sqrt{t/pq}|$ (if $m\geq 1$, in the lower bound
we use that $|t|> |pq|^{1/2(n+m)}$). 
In short, the conditions define a non-empty open subset of the space of parameters if $n\leq m\leq n+1$. 

We can now prove the theorem using induction for $|m-n|\leq 1$. Indeed we start with the case $n=1$, $m=0$ and have the two induction steps $T(n-1,n) \to T(n,n)$ and $T(n,n) \to T(n+1,n)= T(n,n+1)$ (where $T(m,n)$ denotes the statement that the theorem holds for $m$ and $n$). In the base case  $n=1$, $m=0$ we find that our equation corresponds to the univariate case of the evaluation of Theorem \ref{thmdixon}. Note that
the case $n=m=1$ is an iteration of the univariate transformation of Theorem \ref{thmdixon}.

To prove the theorem for arbitrary $n$ and $m$ we can use the idea (already mentioned in \cite{REL}) to take limits from the case $n=m$, which should give an implication $T(n,m) \to T(n,m-1)$. By induction we would then have shown that the theorem holds for all pairs $(n,m)$. 

Taking the limit $v_0 \to t^{-1} t_3$ (while keeping $v_0v_1$ fixed) of the left hand side of the equation, we see that it reduces (using the reflection equation \eqref{eqrefl} for the elliptic gamma function) to 
\[
P_n  \Gamma(t)^n \int_{C^n} 
\Delta_{\II}^{(n)}(t;z) 
\prod_{i=1}^n \prod_{r=0}^2 \Gamma(t_r z_i^{\pm 1}) \Gamma(t^{-1} t_3 z_i^{\pm 1})\prod_{r=2}^{2k-1} \Gamma(v_r z_i^{\pm 1}) \frac{dz_i}{2\pi i z_i} ,
\]
that is the same kind of integral with $t_3$ replaced by $t^{-1} t_3$ and $k$ reduced by 1. On the right hand side the factor
$\Gamma(t_3v_1)$ converges to $\Gamma(pq)=0$, whereas the integral becomes singular. 
In order to calculate the limit we want to do some residue calculus, quite similar to that leading to 
\cite[Proposition 7.1]{vDS}. 
We prefer to first break symmetry in the integral, which we can do using the following lemma \cite[Lemma 6.2]{Rtrafo}
\begin{lemma}\label{lemsym}
Suppose we have the balancing condition $t^{n-1}u_0u_1u_2u_3=p$ then
\begin{multline*}
\sum_{\sigma \in \{\pm 1\}^n}
\prod_{1\leq i<j\leq n} \frac{\theta(tz_i^{\sigma_i} z_j^{\sigma_j};p)}{\theta(z_i^{\sigma_i}z_j^{\sigma_j};p)}
\prod_{i=1}^n \frac{\theta(u_0z_i^{\sigma_i},u_1z_i^{\sigma_i},u_2z_i^{\sigma_i},u_3z_i^{\sigma_i};p)}{\theta(z_i^{2\sigma_i};p)}
\\ =
\prod_{i=0}^{n-1} \theta(t^i u_0u_1,t^i u_0u_2,t^i u_0u_3;p) = 
\prod_{i=0}^{n-1} \theta(t^i u_0u_1,t^i u_0u_2,t^i u_1u_2;p).
\end{multline*}
\end{lemma}
We can now multiply the integrand on the right hand side with this constant factor (with $n$ replaced by $m$). As the original integrand
is $z_i\to 1/z_i$ symmetric, we see that (after we interchange the finite sum with the integral) each term in the sum is identical. So we can just take one term and multiply that by $2^m$. We can subsequently replace $u_0$ by $v_0$ (which cancels the poles of the integrand at 
$z_i=p^{-k}q^{-l}/v_0$, $k,l\geq 0$), and simplify the integrand to get (with the extra balancing condition $t^{m-1}v_0u_1u_2u_3=q$)
%
\begin{align*} 
RHS &=
\frac{1}{\prod_{i=0}^{m-1} \theta(t^i v_0u_1,t^i v_0u_2,t^i v_0u_3;q)}
\prod_{i=m+1}^n \prod_{0\leq r<s\leq 3} \Gamma(t_rt_s t^{n-i}) 
\prod_{i=0}^{2k-1} \prod_{r=0}^3 \Gamma(t_r v_i)
\\ & \qquad \times 2^m P_{m} \Gamma(t)^{m} \int_{C^{m}} 
\frac{\prod_{1\leq i<j\leq m} \Gamma(pt z_iz_j, tz_i/z_j,tz_j/z_i,t/z_iz_j) }{\prod_{1\leq i<j\leq m} \Gamma(pz_iz_j, z_i/z_j,z_j/z_i,1/z_iz_j) \prod_{i=1}^m \Gamma(pz_i^{2},1/z_i^2)}  
\\ & \qquad \times \prod_{i=1}^{m} \prod_{r=0}^3 \Gamma(t t_r^{-1} z_i^{\pm 1}) \prod_{r=1}^{2k-1} \Gamma(v_r z_i^{\pm 1})
\Gamma(pv_0 z_i, v_0/z_i) \theta(u_1z_i,u_2z_i,u_3z_i;q)
 \frac{dz_i}{2\pi i z_i},
\end{align*}
Now let us assume the parameters satisfy $|pq|<|t|<|t_r|<1$, $|v_r|<1$ (for $r>0$) and $1<|v_0|<1/|t|, 1/|p|,1/|q|$. In this case we can take the contour to be identical for all $z_i$ and equal to some deformation of the unit circle, which remains inside the annulus $\{z~|~ 1\leq |z|<1/|t|\}$ and which includes the pole at $v_0$, while excluding the poles at $t_r/t$ and $1/v_r$ ($r>0$). Note that we can't just take the limit $v_0 \to t_3/t$, because in that case no contour satisfying the stated conditions exists. However we can change the contours of the  integrals to the unit circle by picking up the residue at $v_0$. Once we picked up the residue in one integral, for the remaining integrals $v_0$ is no longer a pole, so we only need to pick up the residue once. All of the $m$ choices of picking up a residue give the same result, so we just get a factor $m$ in front of the integral. 
The result is that 
\begin{align*} 
RHS &=
\frac{1}{\prod_{i=0}^{m-1} \theta(t^i v_0u_1,t^i v_0u_2,t^i v_0u_3;q)}
\prod_{i=m+1}^n \prod_{0\leq r<s\leq 3} \Gamma(t_rt_s t^{n-i}) 
\prod_{i=0}^{2k-1} \prod_{r=0}^3  \Gamma(t_r  v_i)
\\ & \qquad \times 2^m P_{m} \Gamma(t)^{m} \int_{C^{m}} 
\frac{\prod_{1\leq i<j\leq m} \Gamma(pt z_iz_j, tz_i/z_j,tz_j/z_i,t/z_iz_j) }{\prod_{1\leq i<j\leq m} \Gamma(pz_iz_j, z_i/z_j,z_j/z_i,1/z_iz_j) \prod_{i=1}^m \Gamma(pz_i^{2},1/z_i^2)}  
\\ & \qquad \times \prod_{i=1}^{m} \prod_{r=0}^3 \Gamma(t t_r^{-1} z_i^{\pm 1}) \prod_{r=1}^{2k-1} \Gamma(v_r z_i^{\pm 1})
\Gamma(pv_0 z_i, v_0/z_i) \theta(u_1z_i,u_2z_i,u_3z_i;q)
 \frac{dz_i}{2\pi i z_i} \\
 &+ 
\frac{1}{\prod_{i=0}^{m-1} \theta(t^i v_0u_1,t^i v_0u_2,t^i v_0u_3;q)}
\prod_{i=m+1}^n \prod_{0\leq r<s\leq 3} \Gamma(t_rt_s t^{n-i}) 
\prod_{i=0}^{2k-1} \prod_{r=0}^3  \Gamma(t_r v_i)
\\ & \qquad \times m 2^m P_{m} \Gamma(t)^{m} \int_{C^{m-1}} 
\frac{\prod_{1\leq i<j\leq m-1} \Gamma(pt z_iz_j, tz_i/z_j,tz_j/z_i,t/z_iz_j) }{\prod_{1\leq i<j\leq m-1} \Gamma(pz_iz_j, z_i/z_j,z_j/z_i,1/z_iz_j) \prod_{i=1}^{m-1} \Gamma(pz_i^{2},1/z_i^2)}  
\\ & \qquad \times \prod_{i=1}^{m-1} \prod_{r=0}^3 \Gamma(t t_r^{-1} z_i^{\pm 1}) \prod_{r=1}^{2k-1} \Gamma(v_r z_i^{\pm 1})
\Gamma(pv_0 z_i, v_0/z_i) \theta(u_1z_i,u_2z_i,u_3z_i;q)
\\ & \qquad \times 
\frac{\prod_{1\leq i\leq m-1} \Gamma(pt z_iv_0, tz_i/v_0,tv_0/z_i,t/z_iv_0) }{\prod_{1\leq i\leq m-1} \Gamma(pz_iv_0, z_i/v_0,v_0/z_i,1/z_iv_0)  \Gamma(pv_0^{2},1/v_0^2)}  
\\ & \qquad \times \prod_{r=0}^3 \Gamma(t t_r^{-1} v_0^{\pm 1}) \prod_{r=1}^{2k-1} \Gamma(v_r v_0^{\pm 1})
\Gamma(pv_0^2 , v_0/z_i) \theta(u_1v_0,u_2v_0,u_3v_0;q) (p;p)(q;q)
 \frac{dz_i}{2\pi i z_i} 
\end{align*}
where all the contours are now unit circles (even in the first term, where the standard conditions on the contour would require the contour to envelop the pole at $v_0$). In the resulting expression we can take the limit $v_0 \to t_3/t$ (while writing $v_1 = pq/tv_0$) without problems (indeed, it just amounts to plugging this value in). The first term (i.e. the term where we did not pick up any residues) vanishes as the prefactor $\Gamma(v_1t_3)$ becomes $\Gamma(pq)=0$. Before we take the limit in the second term we have to use the 
reflection equation of the elliptic gamma function to write $\Gamma(t_3v_1) \Gamma(tv_0/t_3)=1$, so this prefactor does not arise there. 
We obtain (after cleaning up the resulting expression)
\begin{align*} 
RHS(v_0&=t_3/t) =
\frac{1}{\prod_{i=0}^{m-2} \theta(t^{i} t_3u_1,t^{i} t_3u_2,t^{i} t_3u_3;q)}
\prod_{i=m+1}^n \prod_{0\leq r<s\leq 3} \Gamma(t_rt_s t^{n-i}) 
\\& \qquad \times 
\prod_{i=2}^{2k-1} \prod_{r=0}^3  \Gamma(t_r v_i/t) 
\prod_{r=0}^2 \Gamma(t_rt_3/t,t^2/t_3 t_r) 
\\ & \qquad \times 2^{m-1} P_{m-1} \Gamma(t)^{m-1} \int_{C^{m-1}} 
\frac{\prod_{1\leq i<j\leq m-1} \Gamma(pt z_iz_j, tz_i/z_j,tz_j/z_i,t/z_iz_j) }{\prod_{1\leq i<j\leq m-1} \Gamma(pz_iz_j, z_i/z_j,z_j/z_i,1/z_iz_j) \prod_{i=1}^{m-1} \Gamma(pz_i^{2},1/z_i^2)}  
\\ & \qquad \times \prod_{i=1}^{m-1} \prod_{r=0}^2 \Gamma(t t_r^{-1} z_i^{\pm 1}) \Gamma(t^2/t_3 z_i^{\pm 1}) \prod_{r=2}^{2k-1} \Gamma(v_r z_i^{\pm 1})
 \theta(t_3z_i,u_1z_i,u_2z_i,u_3z_i;q)
 \frac{dz_i}{2\pi i z_i} 
\end{align*}
If we symmetrize this expression (using again Lemma \ref{lemsym}, and the inverse of the method described before to desymmetrize), we 
get the right hand side of the theorem with $m$ replaced by $m-1$, $t_3$ by $t_3/t$ and $v_0$ and $v_1$ removed. 
Thus, given that the theorem holds for $(n,m)$, we have now shown it also for $(n,m-1)$ (for an open set of parameters, so by analytic extension for all parameters). By induction the theorem holds for all values of $(n,m)$.
\end{proof}

\section{Some basic hypergeometric limits}
In this section we give three limits as $p\to 0$ of the new identity, obtaining basic hypergeometric integral identities. 

Let us first define several basic hypergeometric versions of the cross terms $\Delta_{\II}$ and the constants $P$:
\[
\tilde P_{BC_n} := \frac{(q;q)^n}{2^n n!}, \qquad \tilde P_{A_{n-1}} := \frac{(q;q)^n}{n!}
\]
and 
\begin{align*}
\tilde \Delta_{\II}^{(n)}(t;z) &= \prod_{1\leq i<j\leq n} \frac{(z_i^{\pm 1}z_j^{\pm 1};q)}{(tz_i^{\pm 1}z_j^{\pm 1};q)}
\prod_{i=1}^n (z_i^{\pm 2};q) \\
\tilde \Delta_{I\!I\!I}^{(n)}(t;z) &= \prod_{1\leq i<j\leq n}
\frac{(qz_iz_j/t, z_i/z_j,z_j/z_i;q)}{(tz_iz_j,tz_i/z_j,tz_j/z_i;q)}
\prod_{1\leq i \leq j\leq n} (1-z_iz_j)  \\
\tilde \Delta_{A\II}^{(n)}(t;z) &= \prod_{1\leq i<j\leq n}
\frac{(z_i/z_j,z_j/z_i;q)}{(tz_i/z_j,tz_j/z_i;q)}
\end{align*}

\begin{cor}\label{corbh1}
Under the balancing conditions 
\[
t_0t_1t_2t_3=t^{2+m-n},  \qquad k \leq m+n
\]
 we have 
\begin{align*}
\tilde P_{BC_n} & \frac{1}{(t;q)^n} \int_{C^n} 
\tilde \Delta_{\II}^{(n)}(t;z) 
\prod_{i=1}^n \prod_{r=0}^3 \frac{1}{(t_rz_i^{\pm 1};q)}  \prod_{r=0}^{k-1} \frac{(t v_r z_i^{\pm 1};q)}{(v_r z_i^{\pm 1};q)} \frac{dz_i}{2\pi i z_i}
\\ & =
\prod_{i=m+1}^n \prod_{0\leq r<s\leq 3} \frac{1}{(t_rt_s t^{n-i};q)} 
\prod_{i=0}^{k-1} \prod_{r=0}^3  \frac{(tv_r/t_r;q)}{(t_r v_r;q)}
\\ & \qquad \times \tilde P_{BC_m} \frac{1}{(t;q)^m}  \int_{C^{m}} 
\tilde \Delta_{\II}^{(m)}(t;z) 
\prod_{i=1}^{m} \prod_{r=0}^3 \frac{1}{(t/t_r z_i^{\pm 1};q)}  \prod_{r=0}^{k-1}  \frac{(t v_r z_i^{\pm 1};q)}{(v_r z_i^{\pm 1};q)}
 \frac{dz_i}{2\pi i z_i},
\end{align*}
where the integration contour is a product of unit circles if $|t|<|t_r|<1$, $|t|<1$, $|v_r|<1$ and we view this as an identity between
the meromorphic extensions of the functions otherwise.
\end{cor}
\begin{proof}
We obtain the case $k =m+n$ by replacing $v_{2i+1}$ by $pq/t v_{2i}$ (as dictated by the conditions) in Theorem \ref{thmmain}, and taking the limit $p\to 0$ while keeping $v_{2i}$ and $t$ and $t_r$ constant. Note that we may replace the limit and the integral (on both sides of the equation) if the contours can be chosen to be unit circles for all small values of $p$ (and hence can remain fixed), as the convergence is uniform on a compact set. There exists an open set of conditions ($|t|<1$, $|t|<|t_r|<1$ and $|v_{2i}|<1$) for which this is possible. By analytic extension the result then holds as an identity between meromorphic functions for all values of the parameters.

Subsequently we can take the limit $v_i\to 0$ for some $i$ to obtain the cases with $k<m+n$.
\end{proof}
Another basic hypergeometric corollary can be obtained by taking a symmetry breaking limit. 
\begin{cor}\label{cor2}
Under the balancing conditions 
\[
\frac{t_0t_1}{s_0s_1}=t^{m-n},\qquad   k= m+n, \qquad  t^{n+1}t_0t_1u_2u_3=q,\qquad  t^{n+1}s_0s_1w_2w_3=q
\]
we have the following equation 
\begin{align*}
\frac{ \tilde P_{A_{n-1}}}{(t;q)^n}&
 \int_{C^n} 
\tilde\Delta_{\III}^{(n)}(t;z) 
\prod_{i=1}^n \frac{(qz_i s_0,qz_i s_1;q)}{(tt_0z_i^{\pm 1}, tt_1z_i^{\pm 1}, z_i/s_0,z_i/s_1;q)}
\prod_{r=0}^{k-1} \frac{(z_i v_r t ,qz_i/v_r,  ;q)}{(v_rz_i, qz_i/tv_r;q)}
\prod_{i=1}^n \theta(u_2z_i,u_3z_i;q)
\frac{dz_i}{2\pi i z_i}
\\ & =
\frac{\prod_{i=1}^{n} \theta(t^{i} t_0u_2,t^{i} t_0u_3;q)}{\prod_{i=1}^{m} \theta(t^{i} w_2 s_0,t^{i} w_3 s_0;q)}
\prod_{i=m+1}^n \frac{(qt^{i-n} s_0s_1;q)}{(t_0t_1t^{2+n-i};q)}
 \prod_{j=0}^1 \prod_{k=0}^1 \frac{1}{(t_jt^{1+n-i}/s_k;q)}
\\ & \qquad \times 
\prod_{r=0}^{k-1} \frac{(qs_0/v_r,tv_r s_0,qs_1/v_r, v_rts_1;q)}{(t_0tv_r, qt_0/v_r,t_1tv_r, qt_1/v_r;q)}
\\ & \qquad \times \frac{\tilde P_{A_{m-1}}}{(t;q)^m} \int_{C^{m}} 
\tilde \Delta_{\III}^{(m)}(t;z) 
\prod_{i=1}^{m} \frac{(qt_0z_i, qt_1z_i;q)}{(z_i/t_0,z_i/t_1,ts_0z_i^{\pm 1}, ts_1z_i^{\pm 1};q)}
\\ & \qquad \qquad \qquad \qquad \qquad  \times \prod_{r=0}^{k-1} \frac{(v_rtz_i,qz_i/v_r;q)}{(v_rz_i, qz_i/tv_r;q)}
\prod_{i=1}^m \theta(w_2z_i,w_3z_i;q)
\frac{dz_i}{2\pi i z_i},
\end{align*}
\end{cor}
And as a further limit we obtain
\begin{cor}\label{cor3}
Under the balancing conditions 
\[
\frac{t_0t_1}{s_0s_1}=t^{m-n},  \qquad k \leq m+n,\qquad  t^{n+1}t_0t_1u_2u_3=q,\qquad  t^{n+1}s_0s_1w_2w_3=q
\]
the following equation holds 
\begin{align*}
\frac{ \tilde P_{A_{n-1}}}{(t;q)^n}&
 \int_{C^n} 
\tilde\Delta_{A\II}^{(n)}(t;z) 
\prod_{i=1}^n \frac{1}{(tt_0/z_i, tt_1/z_i, z_i/s_0,z_i/s_1;q)}
\prod_{r=0}^{k-1} \frac{(z_i v_r t ;q)}{(v_rz_i;q)}
\prod_{i=1}^n \theta(u_2z_i,u_3z_i;q)
\frac{dz_i}{2\pi i z_i}
\\ & =
\frac{\prod_{i=1}^{n} \theta(t^{i} t_0u_2,t^{i} t_0u_3;q)}{\prod_{i=1}^{m} \theta(t^{i} w_2 s_0,t^{i} w_3 s_0;q)}
\prod_{i=m+1}^n  \prod_{j=0}^1 \prod_{k=0}^1 \frac{1}{(t_jt^{1+n-i}/s_k;q)}
\prod_{r=0}^{k-1}  \frac{(tv_r s_0, v_rts_1;q)}{(t_0tv_r,t_1tv_r;q)}
\\ & \qquad \times \frac{\tilde P_{A_{m-1}}}{(t;q)^m} \int_{C^{m}} 
\tilde \Delta_{A\II}^{(m)}(t;z) 
\prod_{i=1}^{m} \frac{1}{(z_i/t_0,z_i/t_1,ts_0/z_i, ts_1/z_i;q)}
\prod_{r=0}^{k-1} \frac{(v_rtz_i;q)}{(v_rz_i;q)}
\prod_{i=1}^m \theta(w_2z_i,w_3z_i;q)
\frac{dz_i}{2\pi i z_i},
\end{align*}
\end{cor}

\begin{proof}[Proof of Corollaries \ref{cor2} and \ref{cor3}]
In order to break the symmetry we use the same argument as in the proof of Theorem \ref{thmmain}, when we wanted to take the limit which reduced $m$.
Now however we specialize the parameters in Lemma \ref{lemsym} as $u_0=t_0$ and $u_1=t_1$ (while we still choose $u_2$ and $u_3$ arbitrarily). 
Moreover we break symmetry on both sides of the equation (but on the right hand side we take the parameters $u$ from Lemma \ref{lemsym} to be $t/t_2$, $t/t_3$, $w_2$ and $w_3$), instead of just on one side.

After simplifying the expression using the difference equation of the elliptic gamma function we obtain that under the balancing conditions
\[
t_0t_1t_2t_3=t^{2+m-n},  \qquad v_{2i}v_{2i+1} = pq t^{-n_i}, \qquad k = m+n,\qquad   t^{n-1}t_0t_1u_2u_3=q,\qquad  t^{n+1}w_2w_3=qt_2t_3
\]
with $n_i \in \mathbb{Z}_{\geq 0}$ we have
\begin{align*}
P_{A_{n-1}} & \Gamma(t)^n 
 \int_{C^n} 
\prod_{1\leq i<j\leq n} \frac{\Gamma(ptz_iz_j,tz_i/z_j,tz_j/z_i,t/z_iz_j)}{\Gamma(pz_iz_j,z_i/z_j,z_j/z_i,1/z_iz_j)}\\ & \qquad \times 
\prod_{i=1}^n \frac{\Gamma(pt_0z_i,t_0/z_i,pt_1z_i,t_1/z_i,t_2z_i^{\pm 1}, t_3 z_i^{\pm 1})\prod_{r=0}^{2k-1} \Gamma(v_r z_i^{\pm 1})  \theta(u_2z_i,u_3z_i;q) }{\Gamma(pz_i^2, z_i^{-2})}
\frac{dz_i}{2\pi i z_i}
\\ & =
\frac{\prod_{i=0}^{n-1} \theta(t^i t_0t_1,t^i t_0u_2,t^i t_0u_3;q)}{\prod_{i=0}^{m-1} \theta(t^{i+2}/t_2t_3,t^{i+1} w_2/t_2,t^{i+1} w_3/t_2;q)}
\prod_{i=m+1}^n \prod_{0\leq r<s\leq 3} \Gamma(t_rt_s t^{n-i}) 
\prod_{i=0}^{2k-1} \prod_{r=0}^3  \Gamma(t_r v_i)
\\ & \qquad \times P_{A_{m-1}} \Gamma(t)^{m} \int_{C^{m}} 
\prod_{1\leq i<j\leq m} \frac{\Gamma(ptz_iz_j,tz_i/z_j,tz_j/z_i,t/z_iz_j)}{\Gamma(pz_iz_j,z_i/z_j,z_j/z_i,1/z_iz_j)} \\ & \qquad  \times
\prod_{i=1}^m \frac{ \Gamma(t/t_0 z_i^{\pm 1}, t/t_1 z_i^{\pm 1}, pt z_i/t_2, t/t_2z_i, ptz_i/t_3, t/t_3z_i)\prod_{r=0}^{2k-1} \Gamma(v_r z_i^{\pm 1})  \theta(w_2z_i,w_3z_i;q) }{\Gamma(pz_i^2, z_i^{-2})}
\frac{dz_i}{2\pi i z_i},
\end{align*}
where the integration contours are unit circles if $|t|<|t_r|<1$ and $|v_r|<1$ and for other parameters we view it as an equation between 
meromorphic functions.

Now we replace $(t_0,t_1,t_2,t_3)$ by $(p^{-1/2}t_0,p^{-1/2}t_1,p^{1/2}t_2,p^{1/2}t_3)$, and 
$v_r$ by $p^{1/2}v_r$. Moreover we set $u_{2/3} = p^{1/2} u_{2/3}$ and $w_{2/3}= p^{1/2}w_{2/3}$. Finally we shift the integration parameters $z_i$ to $p^{-1/2}z_i$ (and move the integration contours back).
Then we obtain under the balancing conditions 
\[
t_0t_1t_2t_3=t^{2+m-n},  \qquad v_{2i}v_{2i+1} = \frac{q}{t}, \qquad k = m+n,\qquad  t^{n-1}t_0t_1u_2u_3=q,\qquad  t^{n+1}w_2w_3=qt_2t_3
\]
with $n_i \in \mathbb{Z}_{\geq 0}$ the equation
\begin{align*}
P_{A_{n-1}} & \Gamma(t)^n 
 \int_{C^n} 
\prod_{1\leq i<j\leq n} \frac{\Gamma(tz_iz_j,tz_i/z_j,tz_j/z_i,pt/z_iz_j)}{\Gamma(z_iz_j,z_i/z_j,z_j/z_i,p/z_iz_j)}\\ & \qquad \times 
\prod_{i=1}^n \frac{\Gamma(t_0z_i^{\pm 1},t_1z_i^{\pm 1},t_2z_i, pt_2/z_i, t_3z_i,pt_3/z_i)\prod_{r=0}^{2k-1} \Gamma(v_rz_i,pv_r/z_i)  \theta(u_2z_i,u_3z_i;q) }{\Gamma(z_i^2, pz_i^{-2})}
\frac{dz_i}{2\pi i z_i}
\\ & =
\frac{\prod_{i=0}^{n-1} \theta(t^i t_0t_1/p,t^i t_0u_2,t^i t_0u_3;q)}{\prod_{i=0}^{m-1} \theta(t^{i+2}/pt_2t_3,t^{i+1} w_2/t_2,t^{i+1} w_3/t_2;q)}
\prod_{i=m+1}^n \Gamma(t_0t_1t^{n-i}/p, pt_2t_3t^{n-i}) \prod_{r=0}^1 \prod_{s=2}^3 \Gamma(t_rt_s t^{n-i}) 
\\ & \qquad \times\prod_{i=0}^{2k-1}  \Gamma(t_0  v_i,t_1  v_i,pt_2  v_i,pt_3  v_i)
P_{A_{m-1}} \Gamma(t)^{m} \int_{C^{m}} 
\prod_{1\leq i<j\leq m} \frac{\Gamma(tz_iz_j,tz_i/z_j,tz_j/z_i,pt/z_iz_j)}{\Gamma(z_iz_j,z_i/z_j,z_j/z_i,p/z_iz_j)} \\ & \qquad  \times
\prod_{i=1}^m \frac{ \Gamma(pt/t_0z_i, tz_i/t_0, pt/t_1z_i, tz_i/t_1, t z_i^{\pm 1}/t_2, tz_i^{\pm 1}/t_3)\prod_{r=0}^{2k-1} \Gamma(v_r z_i,pv_r/z_i)  \theta(w_2z_i,w_3z_i;q) }{\Gamma(z_i^2, pz_i^{-2})}
\frac{dz_i}{2\pi i z_i},
\end{align*}
where the integration contours are again unit circles if $|t|<|t_r|<1$ and $|v_r|<1$. For other parameters we still view this as an equation between meromorphic functions. Using the balancing condition $t_0t_1t_2t_3=t^{2+m-n}$ and the difference equation of the elliptic gamma function we see that 
\[
\frac{\prod_{i=0}^{n-1} \theta(t^i t_0t_1/p;q) \prod_{i=m+1}^n \Gamma(t_0t_1t^{n-i}/p)}{\prod_{i=0}^{m-1} \theta(t^{i+2}/pt_2t_3;q)} =
\prod_{i=m+1}^n \Gamma(t_0t_1t^{n-i})
\]
After this replacement we see that all $\theta$-function in the expression are $p$-independent and the elliptic gamma functions are either of the form $\Gamma(y)$ or $\Gamma(py)$, so we can simply plug in $p=0$ to obtain a basic hypergeometric limit. Simplifying this limit (and in particular replacing all $v_{2r+1}$ by $q/t^{n_r}v_{2r}$) gives us the equation from Corollary \ref{cor2}.

If in the equation of Corollary \ref{cor2} we replace $t_r\to at_r$, $s_r\to as_r$, $v_r \to v_r/a$, $u_r \to u_r/a$, $w_r \to w_r/a$ and $z_i\to az_i$ and shift the integration contour back to a product of unit circles, we can take the limit as $a\to 0$ (in fact, we can just set $a=0$) to obtain the integral identity from Corollary \ref{cor3} for $k=m+n$. Like in the proof of Corollary \ref{corbh1} we can obtain the cases with $k <m+n$ by subsequently setting an appropriate number of $v_i$'s equal to 0.
\end{proof}

\section{Classical limit}\label{secclass}
In this section we will take classical limit of Theorem \ref{thmmain}. That is, we take the limit as $q\to 1$ for appropriate choices of the parameters, to end up with a Selberg like multivariate beta integral. 

The resulting classical integrals were studied before in \cite{Kaneko} and \cite{Yan}. It is shown there that the integrals have series expressions which are a generalization of Gauss' hypergeometric function ${}_2F_1$. In particular the first of the two identities in Corollary \ref{corclass} appears in \cite{Yan}
as a generalization of Euler's transformations for a ${}_2F_1$. The second identity appears to be new.

We need $|m-n|\leq 1$ in order to be able to take the desired limit in the proof. Thus we only have the two cases with $m=n$ and $m+1=n$, which are explicitly given in the corollary below. While we can take a formal limit if
$|m-n|>1$, the resulting integrals do not converge, so we do not see an obvious extension of these results to a transformation between two integrals with completely unrelated numbers of integration variables. 

In this section we write $\Gamma_e$ for the elliptic gamma function, and $\Gamma_c$ for the classical gamma function.
\begin{cor}\label{corclass}
Let $a_0<a_1$ be real and suppose $b_r \in \mathbb{R} \backslash [a_0,a_1]$. Let $0<\Re(\alpha_0), \Re(\alpha_1)$ and
$\tau = (\alpha_0+\alpha_1)/2$
We have
\begin{align*}
\prod_{r=0}^1 & \prod_{i=0}^{2n-1} |a_r - b_i|^{\tau-\alpha_r} 
 \int_{[a_0,a_1]^n} \prod_{1\leq j<k\leq n} 
|x_i-x_j|^{2\tau} \prod_{i=1}^n \frac{|x_i-a_0|^{\alpha_0-1} |x_i-a_1|^{\alpha_1-1}}{\prod_{r=0}^{2n-1} |x_i-b_r|^{\tau}} dx_i
\\ & =
 \int_{[a_0,a_1]^n} \prod_{1\leq j<k\leq n} 
|x_i-x_j|^{2\tau} \prod_{i=1}^n \frac{|x_i-a_0|^{\alpha_1-1} |x_i-a_1|^{\alpha_0-1}}{\prod_{r=0}^{2n-1} |x_i-b_r|^{\tau}} dx_i.
\end{align*}
Under the same conditions as before, but now with balancing condition $\tau = \alpha_0+\alpha_1$, we obtain
\begin{align*} 
\frac{1}{n!} &
 \int_{[a_0,a_1]^n} \prod_{1\leq j<k\leq n} 
|x_i-x_j|^{2\tau} \prod_{i=1}^n \frac{|x_i-a_0|^{\alpha_0-1} |x_i-a_1|^{\alpha_1-1}}{\prod_{r=0}^{2n-2} |x_i-b_r|^{\tau}} dx_i
\\ & =
\frac{\Gamma_c(\alpha_0, \alpha_1)}{\Gamma_c(\tau)} 
\frac{|a_0-a_1|^{\tau-1}}{ \prod_{i=0}^{2n-2} |a_0 - b_i|^{\alpha_1} |a_1 - b_i|^{\alpha_0} }
\\ & \qquad \times 
\frac{1}{(n-1)!} 
 \int_{[a_0,a_1]^{n-1}} \prod_{1\leq j<k\leq n-1} 
|x_i-x_j|^{2\tau} \prod_{i=1}^{n-1}m \frac{|x_i-a_0|^{2\tau-\alpha_0-1} |x_i-a_1|^{2\tau - \alpha_1-1}}{\prod_{r=0}^{2n-2} |x_i-b_r|^{\tau}} dx_i.
\end{align*}
\end{cor}

\begin{proof}
We first prove both identities for $0\leq a_0<a_1$ both real, $0<\Re(\alpha_0), \Re(\alpha_1)$ and $b_r\in \mathbb{R}_{<0}$ (and the appropriate balancing condition involving $\alpha_0$, $\alpha_1$ and $\tau$). The more general parameter conditions stated in the theorem can then be obtained by applying an appropriate linear fractional transformation to the integration variables.

We change the parameters in Theorem \ref{thmmain} to 
$t_0 = q^{\alpha_0^+} a_0$, $t_1 = q^{\alpha_0^-} /a_0$, $t_2=q^{\alpha_1^+}a_1$, $t_3=q^{\alpha_1^-} /a_1$ and
$t=q^{\tau}$ (so the balancing condition becomes $\alpha_0+\alpha_1= (2+m-n)\tau$, where $\alpha_r=\alpha_r^+ + \alpha_r^-$), 
and $v_{2r} = \sqrt{pq} q^{-\beta_r^+}/b_r$, $v_{2r+1} =  \sqrt{pq} q^{-\beta_r^-} b_r$, so $\beta_r^++\beta_r^{-1}=\tau$. Here we take $a_0$, $a_1$ and $b_r$ on the unit circle, with 
$0\leq \arg(a_0)\leq \arg(a_1)\leq \pi$ and we impose the conditions $0<\Re(\alpha_0^{\pm }), \Re(\alpha_1^{\pm})< \Re(\tau)$. Finally we set $q=\exp(2\pi i v w)$ for some $\omega$ in the upper half plane (so $|q|<1$) and $v>0$ real. Then 
using \cite[Theorem 7.4]{Rlimits} we obtain for the left hand side ($LHS$) of Theorem \ref{thmmain} that
\begin{align*}
\lim_{v \to 0^+} &
\prod_{j=0}^{n-1} \frac{\Gamma_e(t^{(n+m-j)})}{\Gamma_e(t^{j+1}) \prod_{0\leq r<s\leq 3} \Gamma_e(t^j t_rt_s)}
LHS \\ &=
| \theta(a_0a_1^{\pm 1};p)|^{n-n(m+1)\tau}
\prod_{j=0}^{n-1} \frac{\Gamma_c((m+n-j)\tau ,\tau)}
{\Gamma_c((j+1)\tau, j\tau+\alpha_0,j\tau+\alpha_1)}
\frac{(2\pi (p;p)^2)^n}{n!} 
\\ & \qquad \times \int_{[a_0,a_1]^n} \prod_{1\leq j<k\leq n} |\theta(z_iz_j^{\pm 1};p)|^{2\tau}
\prod_{i=1}^n \frac{|\theta(a_0z_i^{\pm 1};p)|^{\alpha_0-1} |\theta(a_1z_i^{\pm 1};p)|^{\alpha_1-1}}
{\prod_{r=0}^{k-1} |\theta(p^{1/2} b_r z_i^{\pm 1};p)|^{\tau} }
|\theta(z_i^2;p)| \frac{dz_i}{2\pi i z_i}
\end{align*}
Similarly we see for the right hand side ($RHS$) we see that
\begin{align*}
\lim_{v \to 0^+} &
\frac{1}{\prod_{i=m+1}^n \prod_{0\leq r<s\leq 3}\Gamma_e(t_rt_st^{n-i}) \prod_{i=0}^{2k-1} \prod_{r=0}^3 \Gamma_e(t_rv_i)}
\prod_{j=0}^{m-1} \frac{\Gamma_e(t^{(n+m-j)})}{\Gamma_e(t^{j+1}) \prod_{0\leq r<s\leq 3} \Gamma_e(t^{j+2}/ t_rt_s)}
RHS \\ &=
| \theta(a_0a_1^{\pm 1};p)|^{m-m(n+1)\tau}
\prod_{j=0}^{m-1} \frac{\Gamma_c((m+n-j)\tau ,\tau)}
{\Gamma_c((j+1)\tau, (j+2)\tau-\alpha_0,(j+2)\tau-\alpha_1)}
\frac{(2\pi (p;p)^2)^m}{m!} 
\\ & \qquad \times \int_{[a_0,a_1]^m} \prod_{1\leq j<k\leq m} |\theta(z_iz_j^{\pm 1};p)|^{2\tau}
\prod_{i=1}^m \frac{|\theta(a_0z_i^{\pm 1};p)|^{2\tau-\alpha_0-1} |\theta(a_1z_i^{\pm 1};p)|^{2\tau-\alpha_1-1}}
{\prod_{r=0}^{k-1} |\theta(p^{1/2} b_r z_i^{\pm 1};p)|^{\tau} }
|\theta(z_i^2;p)| \frac{dz_i}{2\pi i z_i}
\end{align*}

Now note the two equations (to prove the second equation we use the balancing condition $t_0t_1t_2t_3=t^{2+m-n}$)
\begin{align*}
\prod_{j=0}^{m-1} \frac{\Gamma(t^{n+m-j})}{\Gamma_e(t^{j+1})} & = 
\prod_{j=0}^{n-1} \frac{\Gamma(t^{n+m-j})}{\Gamma_e(t^{j+1})} \\
\prod_{0\leq r<s\leq 3} \prod_{j=0}^{n-1}  \Gamma_e(t^j t_rt_s) &=
\prod_{0\leq r<s\leq 3} \prod_{i=m+1}^n \Gamma_e(t^{n-i}t_rt_s) \prod_{i=0}^{m-1} \Gamma_e(t^{j+2}/t_rt_s).
\end{align*}
The difference between the two rescaling coefficients is thus given by $\prod_{i,r} \Gamma_e(t_rv_i)$. The limit of this term can be determined using \cite[Theorem 2.13]{Rlimits}. Indeed we have
\[
\lim_{v\to 0^+} \prod_{r=0}^3 \prod_{i=0}^{2k-1} \Gamma_e(t_rv_i) = 
\prod_{r=0}^1 \prod_{i=0}^{k-1} |\theta(a_ib_r^{\pm 1} \sqrt{p};p)|^{\alpha_i-\tau}
\]
Combining the limits we obtain the equation
\begin{align*}
\prod_{r=0}^1 \prod_{i=0}^{k-1} &|\theta(a_ib_r^{\pm 1} \sqrt{p};p)|^{
\tau-\alpha_i}
| \theta(a_0a_1^{\pm 1};p)|^{n-n(m+1)\tau} 
\prod_{j=0}^{n-1} \frac{\Gamma_c((m+n-j)\tau ,\tau)}
{\Gamma_c((j+1)\tau, j\tau+\alpha_0,j\tau+\alpha_1)}
\frac{(2\pi (p;p)^2)^n}{n!} 
\\ & \qquad \times \int_{[a_0,a_1]^n} \prod_{1\leq j<k\leq n} |\theta(z_iz_j^{\pm 1};p)|^{2\tau}
\prod_{i=1}^n \frac{|\theta(a_0z_i^{\pm 1};p)|^{\alpha_0-1} |\theta(a_1z_i^{\pm 1};p)|^{\alpha_1-1}}
{\prod_{r=0}^{k-1} |\theta(p^{1/2} b_r z_i^{\pm 1};p)|^{\tau} }
|\theta(z_i^2;p)| \frac{dz_i}{2\pi i z_i}
\\ & =
| \theta(a_0a_1^{\pm 1};p)|^{m-m(n+1)\tau}
\prod_{j=0}^{m-1} \frac{\Gamma_c((m+n-j)\tau ,\tau)}
{\Gamma_c((j+1)\tau, (j+2)\tau-\alpha_0,(j+2)\tau-\alpha_1)}
\frac{(2\pi (p;p)^2)^m}{m!} 
\\ & \qquad \times \int_{[a_0,a_1]^m} \prod_{1\leq j<k\leq m} |\theta(z_iz_j^{\pm 1};p)|^{2\tau}
\prod_{i=1}^m \frac{|\theta(a_0z_i^{\pm 1};p)|^{2\tau-\alpha_0-1} |\theta(a_1z_i^{\pm 1};p)|^{2\tau-\alpha_1-1}}
{\prod_{r=0}^{k-1} |\theta(p^{1/2} b_r z_i^{\pm 1};p)|^{\tau} }
|\theta(z_i^2;p)| \frac{dz_i}{2\pi i z_i}
\end{align*}
Using the coordinate transformation $x_i=\phi(z_i) = -\frac{\theta(z_i;p)^2}{\theta(-z_i;p)^2}$ as in the discussion after \cite[Theorem 7.2]{Rlimits} we obtain the identity
\begin{align*}
|\phi(a_0)&-\phi(a_1)|^{(n-m)(1-\tau)}  \prod_{r=0}^1 \prod_{i=0}^{k-1} 
|\phi(a_i) -\phi(\sqrt{p} b_r)|^{\tau-\alpha_i} 
\\& \qquad \times 
\prod_{j=0}^{n-1} \frac{\Gamma_c((m+n-j)\tau ,\tau)}
{\Gamma_c((j+1)\tau, j\tau+\alpha_0,j\tau+\alpha_1)}
\\ & \qquad \times 
\frac{1}{n!} 
 \int_{[\phi(a_0),\phi(a_1)]^n} \prod_{1\leq j<k\leq n} 
|x_i-x_j|^{2\tau} \prod_{i=1}^n \frac{|x_i-\phi(a_0)|^{\alpha_0-1} |x_i-\phi(a_1)|^{\alpha_1-1}}{\prod_{r=0}^{k-1} |x_i-\phi(p^{1/2}b_r)|^{\tau}} dx_i
\\ & =
\prod_{j=0}^{m-1} \frac{\Gamma_c((m+n-j)\tau ,\tau)}
{\Gamma_c((j+1)\tau, (j+2)\tau-\alpha_0,(j+2)\tau-\alpha_1)}
\\ & \qquad \times 
\frac{1}{m!} 
 \int_{[\phi(a_0),\phi(a_1)]^m} \prod_{1\leq j<k\leq m} 
|x_i-x_j|^{2\tau} \prod_{i=1}^m \frac{|x_i-\phi(a_0)|^{2\tau-\alpha_0-1} |x_i-\phi(a_1)|^{2\tau - \alpha_1-1}}{\prod_{r=0}^{k-1} |x_i-\phi(p^{1/2}b_r)|^{\tau}} dx_i
\end{align*}
Simplifying this expression (and renaming the $\phi(a_r)$ and $\phi(p^{1/2}b_r)$'s) gives the desired expressions. Note that $\phi(z)\geq 0$ if $z=e^{i\theta}$ for $0\leq \theta<\pi$, while $\phi(p^{1/2}z)< 0$ if $|z|=1$, which gives the parameter conditions. 
\end{proof}

\end{document}